\documentclass[12pt]{article}
\usepackage{amsmath}
\usepackage{latexsym}
\usepackage{amssymb}
\usepackage{a4wide}
\newtheorem{thm}{Theorem}

\newtheorem{Defn}{Definition}
\newtheorem{Remark}{Remark}
\newtheorem{Note}{Note}

\newtheorem{Example}{Example}
\newtheorem{Examples}{Examples}
\newtheorem{Problems}{Problems}

\newtheorem{Problem}{Problem}
\newtheorem{Number}{\!\!}

\newenvironment{proof}{{\noindent\bf Proof.}}%
                  {\nopagebreak\hspace*{\fill}$\Box$\medskip\medskip\par}   

\newcommand{\wb}{\overline}

\newcommand{\mto}{\mapsto}

\newcommand{\R}{{\mathbb R}}

\newcommand{\cN}{{\mathcal N}}

\newcommand{\cL}{{\mathcal L}}
\newcommand{\cM}{{\mathcal M}}

\newcommand{\sub}{\subseteq}

\begin{document}
\begin{center}
{\Large \bf Simplified Proofs for the
Pro-Lie Group Theorem\\[2mm]
and the One-Parameter Subgroup Lifting
Lemma}\\[6mm]
{\bf Helge Gl\"{o}ckner}\vspace{5mm}
\end{center}
\begin{abstract}
This note is devoted to the theory of
projective limits of finite-dimensional
Lie groups, as
developed in the recent monograph
[Hofmann, K.\,H. and S.\,A. Morris,
``The Lie Theory of Connected Pro-Lie Groups,''
EMS Publ.\ House, 2007].
We replace the original, highly non-trivial
proof of the One-Parameter Subgroup
Lifting Lemma given in the monograph
by a shorter and more elementary argument.
Furthermore, we shorten (and correct)
the proof of the so-called
Pro-Lie Group Theorem, which asserts
that pro-Lie groups and
projective limits of Lie groups coincide.\vspace{2mm}
\end{abstract}
\noindent
By a famous theorem of Yamabe \cite{Yam},
every identity neighbourhood of a connected (or almost connected)
locally compact group $G$ contains a
closed normal subgroup $N$ such that $G/N$ is a Lie group,
and thus is a so-called pro-Lie group.
Therefore locally compact pro-Lie groups
form a large class of locally compact groups,
which has been studied by many authors
(see, e.g., \cite{Iwa}, \cite{Las},
\cite{MaZ} as well as
\cite{HMS} and the references therein).
Although a small number of papers
broached on the topic of non-locally compact pro-Lie groups
(like \cite{Hof} and \cite{Glo}),
a profound structure theory of such groups was only
begun recently in \cite{HMo}
and then fully worked out in
the monograph \cite{HaM}.
The novel results accomplished in \cite{HaM}
make it clear that the study of general pro-Lie groups
is fruitful also for the theory of locally
compact groups.\\[2.5mm]
We recall from \cite{HaM}:
For $G$ a Hausdorff topological group,
$\cN(G)$ denotes the set of all closed
normal subgroups $N$ of $G$ such that $G/N$ is a (finite-dimensional)
Lie group. If $G$ is complete and $\cN(G)$ is a filter basis
which converges to $1$, then $G$ is called
a \emph{pro-Lie group}.
It is easy to see that every pro-Lie group is,
in particular, a projective limit of Lie groups.
Various results which are known in the locally
compact case become much more complicated
to prove for non-locally compact pro-Lie groups.
For example, it is not too hard to see
that
every locally compact group which is a projective
limit of Lie groups is a pro-Lie group
(see
\cite{App} for an elementary argument; the appeal to
the solution of Hilbert's fifth problem
in the earlier proof in \cite{HMS} is unnecessary).
Also, it has been known for a long time \cite{HWY}
that one-parameter subgroups
can be lifted over quotient morphisms $q\colon G\to H$
between
locally compact groups, i.e.,
for each continuous homomorphism $X\colon \R\to H$
there exists a continuous homomorphism
$Y\colon \R\to G$ such that $X=q\circ Y$.
The original proofs for analogues of
the preceding two
results for general pro-Lie groups
as given in \cite{HMo} and \cite{HaM}
(called the ``Pro-Lie Group Theorem'' and ``One-Parameter
Subgroup Lifting Lemma'' there)
were quite long and complicated.
Later, A.\,A. George Michael
gave a short alternative proof of the
Pro-Lie Group Theorem,
which however was not self-contained
but depended on a non-elementary result
from outside, the
Gleason--Palais Theorem:
\emph{If $G$ is a locally arcwise connected topological
group in which the compact metrizable
subsets are of bounded dimension,
then $G$ is a Lie group} \cite[Theorem~7.2]{GaP}.\\[2.5mm]
The goal of this note is to record
two short and simple arguments, which together
with some
10 pages of external reading\footnote{Lemmas 3.20--3.24, Propositions 3.27
and 3.30,
Lemma 3.31 and Lemmas 4.16--4.18 in \cite{HaM}.}
provide elementary and essentially self-contained
proofs for both the Pro-Lie Group
Theorem and the One-Parameter Subgroup
Lifting Lemma
(up to well-known facts).
In this way,
the proof of the latter shrinks from
over 3 pages to 8 lines,
and the proof of the former by 6 pages.
Moreover, the author noticed that the proof of the
Pro-Lie Group Theorem in \cite{HaM} (and \cite{HMo})
depends on an incorrect assertion,\footnote{Parts (iii) and (iv) of the
``Closed Subgroup Theorem''
\cite[Theorem 1.34]{HaM} are false, as the example $G={\mathbb R}$,
$H={\mathbb Z}$, ${\mathcal N}=
\{\{0\},\sqrt{2} \, {\mathbb Z}\}$ shows.
This invalidates the proof of part (iii) of the ``First Fundamental
Lemma'' \cite[Lemma 3.29]{HaM}, which is used
in \cite{HaM} to prove the Pro-Lie Group Theorem
(the proof of Lemma 3.29 (iv) also seems to be defective,
because elements $M\in \cM$
are of the form $M=\ker(f_j)\cap G_0$,
rather than $M=\ker(f_j)$).}
making it the more important to
have a correct elementary proof
available.\\[2.5mm]
Let us now re-state and prove
the theorem and lemma in contention.
Notations from \cite{HaM}
will be used without explanation.
\begin{thm}[The Pro-Lie Group Theorem]
Every projective limit of Lie groups
is a pro-Lie group.
\end{thm}
\begin{proof}
Let $G$ be a projective limit of
a projective system $((G_j)_{j\in J},(f_{jk})_{j\leq k})$
of Lie groups $G_j$ and morphisms $f_{jk}\colon G_k\to G_j$.
By \cite[Proposition 3.27]{HaM}, $G$
will be a pro-Lie group if we can show that
$G/\ker(f_j)$
is a Lie group for each limit map
$f_j\colon G\to G_j$. Let $H_j$ be the analytic subgroup
of $G_j$ with Lie algebra $\cL(f_j)(\cL(G))$
(equipped with its Lie group topology).
By \cite[Lemmas 3.23 and 3.24]{HaM}, $f_j$ restricts
and corestricts to a quotient morphism $\phi_j\colon G_0\to H_j$.
Given $g\in G$, write $I_g^G\colon G\to G$, $I_g^G(h):=ghg^{-1}$.
Since $\phi_j\circ I_g^G|_{G_0}=I_{f_j(g)}^{G_j}\circ \phi_j$,
we see that $I_{f_j(g)}^{G_j}(H_j)\subseteq H_j$ and
$I_{f_j(g)}^{G_j}|_{H_j}\colon H_j\to H_j$ is continuous.
Hence $Q_j:=f_j(G)$ can be made a Lie group with $H_j$
as an open subgroup. Then the corestriction $q_j\colon G\to Q_j$
of $f_j$ to $Q_j$ is a surjective homomorphism, which is open
since so is $f_j|_{G_0}^{H_j}=\phi_j$. If we can show that $q_j$
is continuous, then $q_j$ will be a quotient morphism
and thus $G/\ker(f_j)\cong Q_j$ a Lie group.
However, by \cite[Lemma 3.21]{HaM}, there exists some
$k\in I$ such that $k\geq j$ and $f_{jk}((G_k)_0)\sub
H_j$. Also, it is shown in the proof
of \cite[Lemma 3.24]{HaM} that 
the map
$\wb{f}_{jk}\colon (G_k)_0\to H_j$, $x\mto f_{jk}(x)$
is continuous. Since $U:=f_k^{-1}((G_k)_0)$
is a neighbourhood \mbox{of $1$} in $G$
and $q_j|_U\!=\!\wb{f}_{jk}\!\circ\! f_k|_U^{(G_k)_0}$
is continuous, the homomorphism $q_j$ is
continuous.
\end{proof}
\begin{thm}[The One-Parameter Subgroup Lifting Lemma]
Let $G$ and $H$ be pro-Lie groups and
$f\colon G\to H$ be a quotient morphism
of topological groups.
Then every one-parameter subgroup $X$
of $H$ lifts to one of $G$, i.e., 
there exists a one-parameter subgroup
$Y\colon \R\to G$ such that $X=f \circ Y$.
\end{thm}
\begin{proof}
We adapt an argument from \cite[p.\,193]{HaM}.
By Lemmas 4.16, 4.17 and 4.18 in \cite{HaM},
we may assume that $H=\R$ and have to show
that $f$ is a retraction. If $f$ was not a retraction,
then we would have $\cL(f)(\cL(G))=\{0\}$
and hence $f(G_0)=\{1\}$, using that
$\exp_G(\cL(G))$ generates a dense
subgroup of $G_0$
(by Lemma 3.24 and the proof of Lemma 3.22 in \cite{HaM}),
and $f\circ \exp_G=\exp_H\circ \, \cL(f)=1$.
Hence $f$ factors to a quotient morphism
$G/G_0\to \R$.
Since $G/G_0$ is proto-discrete
by \cite[Lemma 3.31]{HaM}, it would follow
that also its quotient $\R$ is proto-discrete
(see \cite[Proposition 3.30\,(b)]{HaM})
and hence discrete (as $\R$ has no small
subgroups). We have reached a contradiction.
\end{proof}
We mention that the Pro-Lie Group Theorem
has no analogue for projective limits
of Banach-Lie groups. In fact, consider
a Fr\'{e}chet space $E$
which is not a Banach space but admits
a continuous norm $\|.\|$ (e.g., $E=C^\infty([0,1],\R)$).
Then $E$ is a projective limit
of Banach spaces. The $\|.\|$-unit ball $U$ is a $0$-neighbourhood
in $E$ which does not contain any non-trivial subgroup
of~$E$.
If there existed a quotient morphism
$q\colon E\to G$ to a Banach-Lie group $G$
with kernel in $U$, then we would have $\ker(q)=\{0\}$.
Hence $q$ would be an isomorphism,
entailing that the Banach-Lie group~$G$
is abelian and simply connected and
therefore isomorphic to the additive
group of a Banach space. Since~$E$
is not a Banach space, we have reached a contradiction.
\noindent
{\footnotesize
{\bf Helge Gl\"{o}ckner}, TU~Darmstadt, FB~Mathematik~AG~AGF,
Schlossgartenstr.\,7, 64289 Darmstadt,\\
Germany.
E-Mail: {\tt gloeckner@mathematik.tu-darmstadt.de}}
\end{document}